\documentclass[10pt]{amsart}

\usepackage{amsmath}
  \usepackage{paralist}
  \usepackage{graphics} 
  \usepackage{epsfig} 
\usepackage{graphicx}  \usepackage{epstopdf}

 \usepackage[colorlinks=true]{hyperref}
\hypersetup{urlcolor=blue, citecolor=red}

\usepackage{tikz}
\usetikzlibrary{arrows, automata}
\usetikzlibrary{calc}
\usetikzlibrary{intersections}

\tikzset{
    right angle quadrant/.code={
        \pgfmathsetmacro\quadranta{{1,1,-1,-1}[#1-1]}     
        \pgfmathsetmacro\quadrantb{{1,-1,-1,1}[#1-1]}},
    right angle quadrant=1, 
    right angle length/.code={\def\rightanglelength{#1}},   
    right angle length=1ex, 
    right angle symbol/.style n args={3}{
        insert path={
            let \p0 = ($(#1)!(#3)!(#2)$) in     
                let \p1 = ($(\p0)!\quadranta*\rightanglelength!(#3)$), 
                \p2 = ($(\p0)!\quadrantb*\rightanglelength!(#2)$) in 
                let \p3 = ($(\p1)+(\p2)-(\p0)$) in  
            (\p1) -- (\p3) -- (\p2)
        }
    }
}

\newtheorem{theorem}{Theorem}[section]
\newtheorem{corollary}[theorem]{Corollary}

\newtheorem{lemma}[theorem]{Lemma}
\newtheorem{proposition}[theorem]{Proposition}
\newtheorem{conjecture}[theorem]{Conjecture}

\theoremstyle{definition}
\newtheorem{definition}[theorem]{Definition}
\newtheorem{remark}[theorem]{Remark}

\title[Running heading with forty characters or less]
      {On the ergodicity of geodesic flows on surfaces without focal points}

\author[first-name1 last-name1 and first-name2 last-name2]{Weisheng Wu, Fei Liu and Fang Wang}


\address{Department of Applied Mathematics, College of Science, China Agricultural University, Beijing, 100083, China}
 \email{wuweisheng@cau.edu.cn}

\address{College of Mathematics and System Science, Shandong University of Science and Technology, Qingdao, 266590, China.}
\email{feiliu.sdust$@$gmail.com}

\address{School of Mathematical Sciences, Capital Normal University, Beijing, 100048, China; and Beijing Center for Mathematics and Information Interdisciplinary Sciences (BCMIIS), Beijing 100048, China.}
\email{fangwang@cnu.edu.cn}

\begin{document}
\maketitle

\markboth{Ergodicity of geodesic flows}
{Ergodicity of geodesic flows}
\renewcommand{\sectionmark}[1]{}

\begin{abstract}
In this article, we study the ergodicity of the geodesic flows on surfaces with no focal points. Let $M$ be a smooth connected and closed surface equipped with a $C^\infty$ Riemannian metric $g$, whose genus $\mathfrak{g} \geq 2$. Suppose that $(M,g)$ has no focal points. We prove that the geodesic flow on the unit tangent bundle of $M$ is ergodic with respect to the Liouville measure, under the assumption that the set of points on $M$ with negative curvature has at most finitely many connected components.
\end{abstract}

\maketitle
\section{Introduction}\label{section 1}
Assume that $(M,g)$ is a smooth, connected and closed manifold equipped with a $C^\infty$ Riemannian metric $g$. The geodesic flow $g^t$, generated by the Riemannian metric $g$, is defined on the unit tangent bundle $SM$ by the formula:
\[g^t(v)=\gamma'_v(t),\]
where $\gamma'_v(t)$ is the unit vector tangent to the geodesic $\gamma_v(t)$ uniquely determined by the initial vector $v\in SM$. In this paper, we study the ergodicity of the geodesic flow with respect to the Liouville measure $\nu$ on $SM$, where $(M,g)$ is assumed to be a surface having no focal points.\\

Our work was originally inspired by the classical results on the ergodicity of the geodesic flows on Riemannian manifolds with nonpositive curvature. The geodesic flows on Riemannian manifolds with negative or nonpositive curvature have very rich dynamics and broad applications. In the last century, this class of geodesic flows have always been attracting great interests of the mathematicians in dynamical systems and related areas. A lot of beautiful results on the dynamics of the geodesic flows have been exhibited. Among which, the ergodic properties, such as the ergodicity and the mixing properties, the measure of maximal entropy etc, have their special importance and receive extensive attentions. The statistical properties of geodesic flows on surfaces with negative curvature were first studied by Hadamard and Morse in the beginning of the twentieth century. Hopf (cf. \cite{Ho0,Ho1}) proved the ergodicity of the geodesic flow with respect to the Liouville measure $\nu$ on $SM$ for compact surfaces of variable negative curvature and for compact manifolds of constant negative sectional curvature in any dimension. The general case for compact manifolds of variable negative curvature was established by Anosov and Sinai (cf. \cite{An, AnS}). The geodesic flows on compact manifolds of negative curvature is a primary example of the \emph{uniformly hyperbolic flows} (or Anosov flows). Its ergodicity was established based on the classical Hopf argument and results in hyperbolic geometry (cf., for example the appendix in \cite{BB}). \\

Geodesic flows on manifolds of nonpositive curvature have also been intensively studied since 1970's. However, even for surfaces of nonpositive curvature, the geodesic flows present certain \emph{non-uniformly hyperbolic} behaviors. The ergodicity for the geodesic flows faces a great challenge due to the existence of ``flat'' geodesics. Consider a closed surface $M$ of genus $\mathfrak{g} \geq 2$ and of nonpositive curvature. Let
\begin{equation*}\label{e:flat}
\begin{aligned}
\Lambda:=\{v\in SM: K(\gamma_v(t))\equiv 0, \ \forall t\in \mathbb{R}\},
\end{aligned}
\end{equation*}
where $K$ denotes the curvature of the point. We call $\gamma_v$ a \emph{flat geodesic} if $v\in \Lambda$, i.e., $\gamma_v$ is a flat geodesic if the curvature along it is constantly $0$. People still do not know if $\Lambda$ is small in measure ($\nu(\Lambda)=0$ or not), in general. However, from the dynamical point of view, $\Lambda$ should be a very small set. For example, in \cite{Kn}, Knieper showed the strict inequality for the geodesic flows on rank $1$ manifolds of nonpositive curvature:
$$h(g^1|_{\Lambda})<h(g^1),$$
where $g^1$ is the time one map of the geodesic flow $g^t$, $h$ denotes the topological entropy, and $\Lambda$ denotes the irregular set of the geodesic flow, which is a counterpart of the above defined set in arbitrary dimensions. This means that the geodesic flow restricted on $\Lambda$ has less complexity than the whole geodesic flow. Burns and Gelfert proved that on rank $1$ surfaces of nonpositive curvature, the geodesic flow on $\Lambda$ has zero topological entropy (cf.~\cite{BG}). In higher dimensions, it is possible to have positive entropy on $\Lambda$; an example was given by Gromov (cf. \cite{Gro}).\\

For geodesic flows on rank $1$ surfaces of nonpositive curvature, the orbits inside $\Lambda$ are also believed to have simple behavior. In all the known examples, all the orbits in $\Lambda$ are closed. In a recent survey, Burns asks the question: Does there exist a non-closed flat geodesic? (cf. Question 6.2.1 in \cite{BM}). In this paper, we will show that all flat geodesics are closed on surfaces without focal points, under our assumption. Nevertheless, the most important topic on the set $\Lambda$ is still its Liouville measure (how small it is). People expect that on surfaces with nonpositive curvature, $\Lambda$ should have $0$ Liouville measure (this leads to the ergodicity of the geodesic flows, see \cite{BP}). This is the following well-known conjecture on the ergodicity for geodesic flows on surfaces with nonpositive curvature\footnote{Some experts in the area expect a negative answer to the conjecture. Our results in the paper support the conjecture under an additional assumption.}.
\vspace{.1cm}

\begin{conjecture}[Cf. \cite{RH}]\label{conjecture}
Let $(M,g)$ be a smooth, connected and closed surface of genus $\mathfrak{g}\geq 2$, which has nonpositive curvature. Then all flat geodesics are closed and there are only finitely many homotopy classes of such geodesics. In particular, $\nu(\Lambda)=0$, and hence the geodesic flow on $SM$ is ergodic.
\end{conjecture}
\vspace{.1cm}

We declare that the terminology ``ergodicity" in this paper means the ergodicity with respect the Liouville measure $\nu$ on $SM$. The problem we are considering in this paper is the ergodicity of the geodesic flows on surfaces without focal points. First of all, we give the definition of the focal points.
\vspace{.1cm}

\begin{definition}
Let $(M,g)$ be a Riemannian manifold and $\gamma$ a geodesic on $M$. Points $q=\gamma(t_{0})$ and $p=\gamma(t_{1})$ are called \emph{focal} if there exists a Jacobi field $J$ along $\gamma$ such that $J(t_{0})=0$, $J'(t_{0})\neq 0$ and $\frac{d}{dt}\| J(t)\|^{2}\mid_{t=t_{1}}=0$. The Riemannian manifold $(M,g)$ is said to be \emph{without focal points} if there is no focal points on any geodesic of $M$.
\end{definition}
\vspace{.1cm}

It is not hard to see that the manifolds with nonpositive curvature have no focal points. If $M$ is a surface of genus $1$ and has no focal points, then it must be a flat torus (cf.  \cite{Ho,BI}). Therefore, the geodesic flow on $M$ is obviously not ergodic. However, if $M$ has higher genus, the curvature is allowed to vary. In this paper, we always assume that the surface $M$ we are considering has genus greater than $1$.\\

In 1970's, by using his theory of nonuniform hyperbolicity, Pesin obtained a celebrated result on the ergodicity of the geodesic flows on manifolds without focal points, which satisfy the Uniform Visibility Axiom (cf. Theorem 12.2.12 in \cite{BP}). We are not going to give the explicit definition of the Uniform Visibility Axiom here, but remark that it is satisfied by every closed surface of genus $\mathfrak{g}\geq 1$. To state Pesin's result for surfaces without focal points, we define the sets:
\[\Delta^+=\{v\in SM: \chi(v,\xi)<0 \text{\ for any\ } \xi\in E^+(v)\},\]
\[\Delta^-=\{v\in SM: \chi(v,\xi)>0 \text{\ for any\ } \xi\in E^-(v)\}, \]
\[\Delta=\Delta^+\cap \Delta^-,\]
where $\chi$ denotes the Lyapunov exponents and $E^{\pm}$ denotes the stable and unstable distributions on $SM$ with respect to the geodesic flow respectively. $\Delta$ is called the \emph{regular set} with respect to the geodesic flow. For details, see Section \ref{section 2} below. In \cite{BP}, Pesin proved the following theorem:
\vspace{.1cm}

\begin{theorem}[Pesin, cf. \cite{BP}]\label{pesin} For the geodesic flow on a surface without focal points, we have that
$\nu(\Delta)>0$, and $g^t|_\Delta$ is ergodic.
\end{theorem}
\vspace{.1cm}

Our first result in this paper is the following relation between the regular set $\Delta$ and the set $\Lambda$ of unit vectors tangent to flat geodesics. We remark that all our results are established under \emph{the assumption of no focal points}, so sometimes we omit the statement of this assumption in the following theorems.
\vspace{.1cm}

\begin{theorem}\label{regular}
$\nu(\Lambda^c \setminus \Delta)=0$.
\end{theorem}
\vspace{.1cm}

By Theorems \ref{pesin} and \ref{regular}, if $\nu(\Lambda)=0$, the regular set $\Delta\subset SM$ is a full measure set, and then the geodesic flow is ergodic on $SM$. The condition $\nu(\Lambda)=0$ holds in all the known examples so far. However, it is still not proved, even for the surfaces of nonpositive curvature. A recent progress on this problem was made by the first author Wu in \cite{Wu}. We conclude the main result of \cite{Wu} in the following theorem:
\vspace{.1cm}

\begin{theorem}[Cf. \cite{Wu}]\label{wu}
Let $(M,g)$ be a smooth, connected and closed surface of genus $\mathfrak{g}\geq 2$, which has nonpositive curvature. Suppose that the set $\{p\in M: K(p)<0\}$ has finitely many connected components, then $\nu(\Lambda)=0$. In particular, the geodesic flow is ergodic.
\end{theorem}
\vspace{.1cm}

In this paper, we generalize Theorem \ref{wu} from the setting of surfaces with nonpositive curvature to surfaces without focal points. This means that we are going to prove the ergodicity of the geodesic flows on surfaces which can have positive curvature in a subset. To achieve this goal, we explore the properties of flat geodesics, which are also of independent interest. Among them is the following important result. Here, we let $\text{Per}(g^t)$ denote the set of periodic points of the geodesic flow, and $\mathcal{O}(z)$ denote the orbit of $z$ under the geodesic flow. The following theorem says that non-closed flat orbits can accumulate only on non-closed flat orbits.
\vspace{.1cm}

\begin{theorem}\label{nonclosedset}
$\Lambda \cap (\text{Per\ } (g^t))^c$ is a closed subset of $SM$.
\end{theorem}
\vspace{.1cm}

According to the dichotomy: (1) $\Lambda \subset \text{Per}(g^t)$; \ (2)$\Lambda \cap (\text{Per\ } (g^t))^c \neq \emptyset$, we prove the following two results.
\vspace{.1cm}

\begin{theorem}\label{A}
If $\Lambda \subset \text{Per}(g^t)$, then there is a finite decomposition of $\Lambda$:
$$\Lambda = \mathcal{O}_1 \cup \mathcal{O}_2 \cup \ldots \mathcal{O}_k \cup \mathcal{F}_1\cup \mathcal{F}_2 \cup \ldots \cup \mathcal{F}_l,$$
where each $\mathcal{O}_i, 1\leq i \leq k$, is an isolated periodic orbit and each $\mathcal{F}_j, 1\leq j \leq l$, consists of vectors tangent to a flat strip. Here $k$ or $l$ are allowed to be $0$ if there is no isolated closed flat geodesic or no flat strip.
\end{theorem}
\vspace{.1cm}

We remark that, if $\Lambda \subset \text{Per}(g^t)$, then Theorem \ref{A} immediately implies $\nu(\Lambda)=0$, and therefore the geodesic flow is ergodic.
\vspace{.1cm}

\begin{theorem}\label{B}
If $\Lambda \cap (\text{Per} (g^t))^c \neq \emptyset$, then there exist $y, z \in \Lambda$, $y\notin \mathcal{O} (z)$, such that
$$d(g^t(y), g^t(z))\to 0,  \ \ \text{as\ } t\to +\infty.$$
\end{theorem}
\vspace{.1cm}

Our main result is the following theorem, which means that under certain condition the scenario in Theorem \ref{B} can not happen.
\vspace{.1cm}

\begin{theorem}\label{Main}
If the set $\{p\in M: K(p)<0\}$ has at most finitely many connected components, then $\Lambda \subset \text{Per}(g^t)$. In particular, the geodesic flow is ergodic.
\end{theorem}
\vspace{.1cm}

Theorem \ref{Main} gives a negative answer to Question 6.2.1 asked by Burns in \cite{BM}, for surfaces without focal points when $\{p\in M: K(p)<0\}$ has at most finitely many connected components. Furthermore, Theorem \ref{A} exhibits that in fact there are at most finitely many flat strips and isolated closed flat geodesics in this case.\\

So far, it is still unknown that whether Conjecture \ref{conjecture} is true or not in general. In Section \ref{section 4}, we discover several properties of the flat geodesics on surfaces without focal points, which include:
\begin{itemize}
\item All flat strips are closed.
\item A unit vector not tangent to a flat strip has the expansivity property.
\item An ideal triangle with a flat geodesic as an edge has infinite area.
\item A non-closed flat geodesic has infinitely many self-intersections.
\end{itemize}
All these results together with our Theorem \ref{nonclosedset} are believed to be important toward Conjecture \ref{conjecture} in future research.\\

The paper is organized as follows. In Section \ref{section 2}, we will present some preliminaries on the geodesic flows on surfaces without focal points. The proof of Theorem \ref{regular} is shown afterward in Section \ref{section 3}. In Section \ref{section 4}, we prove Theorem \ref{nonclosedset} and the above properties of the flat geodesics. Our main Theorems \ref{A}, \ref{B} and \ref{Main} are proved in the last section. Throughout the remainder of the paper, we always let $M$ be a smooth, connected and closed surface with genus $\mathfrak{g} \geq 2$, and equipped with a $C^\infty$ Riemannian metric $g$ without focal points. \\

\section{Preliminaries on surfaces without focal points}\label{section 2}

\subsection{Jacobi fields, stable and unstable distributions}

In order to study the dynamics of geodesic flows, we should investigate the geometry of the second tangent bundle $TTM$. Let $\pi:TM\rightarrow M$ be the natural projection, i.e., $\pi(v)=p$ where $v \in T_{p}M$.
The connection map $K_{v}:T_{v}TM\rightarrow T_{\pi(v)}M$ is defined as follows. For any $\xi\in T_vTM$, $K_{v}\xi:= (\nabla X)(t)|_{t=0}$,
where $X:(-\epsilon,\epsilon)\rightarrow TM$ is a smooth curve satisfying $X(0)=v$ and $X'(0)=\xi$, and $\nabla$ is the covariant derivative along the curve $\pi(X(t))\subset M$. Then the standard Sasaki metric on $TTM$ is given by
$$\langle\xi,\eta\rangle_{v}=\langle d\pi_{v}\xi,d\pi_{v}\eta\rangle+\langle K_{v}\xi,K_{v}\eta\rangle, \quad \xi, \eta\in T_vTM.$$
\vspace{-.2cm}

Recall that the \emph{Jacobi equation} along a geodesic $\gamma_v(t)$ is
\begin{equation}\label{e:jacobi0}
J''(t)+R(\gamma'_v(t),J(t))\gamma'_v(t)=0,
\end{equation}
where $R$ is the curvature tensor, and $J(t)$ is a Jacobi field along $\gamma_v(t)$ and perpendicular to $\gamma'_v(t)$. Suppose $J_{\xi}(t)$ is the solution of \eqref{e:jacobi0} which satisfies the initial conditions
\[J_{\xi}(0)=d\pi_{v}\xi, ~~\frac{d}{dt}\Big|_{t=0}J_{\xi}(t)=K_{v}\xi.\]
Then, it follows that (cf. p. 386 in \cite{BP})
\[J_{\xi}(t)=d\pi_{g^{t}v}dg^{t}_{v}\xi, ~~\frac{d}{dt}J_{\xi}(t)=K_{g^{t}v}dg^{t}_{v}\xi.\]
On the surface $M$, we have the Fermi coordinates $\{e_1(t), e_2(t)\}$ along the geodesic $\gamma_v(t)$, obtained by the time $t$-parallel translations along $\gamma_v(t)$ of an orthonormal basis $\{e_1(0), e_2(0)\}$ where $e_1(0)=\gamma'_v(0)$. Thus $e_1(t)=\gamma'_v(t)$ and $e_2(t) \perp \gamma'_v(t)$. Suppose that $J(t)=j(t)e_2(t).$ Then the Jacobi equation \eqref{e:jacobi0} becomes
\begin{equation}\label{e:jacobi}
j''(t)+K(t)j(t)=0,
\end{equation}
where $K(t)=K(\gamma_v(t))$ is the curvature at point $\gamma_v(t)$.
Let $u(t)=j'(t)/j(t)$. Then the Jacobi equation \eqref{e:jacobi} can be written in an equivalent form:
\begin{equation}\label{e:ricatti}
u'(t)+u^2(t)+K(t)=0,
\end{equation}
which is called \emph{the Riccati equation}.\\

Using the Fermi coordinates, we can write \eqref{e:jacobi0} in the matrix form
\begin{equation}\label{e:jacobi1}
\frac{d^2}{dt^2}A(t)+K(t)A(t)=0.
\end{equation}
The following result is a standard fact.
\vspace{.1cm}

\begin{proposition}[Cf. \cite{Eb}]\label{Jacobi}
Given $s\in \mathbb{R}$, let $A_{s}(t)$ be the unique solution of \eqref{e:jacobi1}
satisfying $A_{s}(0)=Id$ and $A_{s}(s)=0$, then there exists a limit
\[A^{+}=\lim_{s\rightarrow +\infty}\frac{d}{dt}\Big|_{t=0}A_{s}(t).\]
\end{proposition}
\vspace{.1cm}

Now we can define the \emph{positive limit solution} $A^+(t)$ as the solution of \eqref{e:jacobi1}
satisfying the initial conditions
\[A^{+}(0)=Id, \quad \frac{d}{dt}\Big|_{t=0}A^{+}(t)=A^{+}.\]
It's easy to see that $A^{+}(t)$ is non-degenerate for all $t\in \mathbb{R}$. Similarly, letting $s \rightarrow -\infty$, one can define the \emph{negative limit solution} $A^{-}(t)$ of \eqref{e:jacobi1}.\\

For each $v \in SM$, define
\[E^{+}(v):=\{\xi \in T_{v}SM: \langle\xi,V(v)\rangle=0 \text{\ and\ } J_{\xi}(t)=A^{+}(t)d\pi_{v}\xi\},\]
\[E^{-}(v):=\{\xi \in T_{v}SM: \langle\xi,V(v)\rangle=0 \text{\ and\ } J_{\xi}(t)=A^{-}(t)d\pi_{v}\xi\},\]
where $V$ is the vector field generated by the geodesic flow and $J_{\xi}$ is the solution of \eqref{e:jacobi0} satisfying
\[J_{\xi}(0)=d\pi_{v}\xi,\ ~~\frac{d}{dt}\Big|_{t=0}J_{\xi}(t)=K_{v}\xi.\]
One can check the following properties of $E^{+}(v)$ and $E^{-}(v)$ (see \cite{BP} for more details).
\vspace{.1cm}

\begin{proposition}[Cf. Proposition 12.1.1 in \cite{BP}]\label{subspace}
$E^{+}(v)$ and $E^{-}(v)$ have the following properties:
\begin{enumerate}
  \item $E^{+}(v)$ and $E^{-}(v)$ are $1$-dimensional subspaces of $T_{v}SM$.
  \item $d\pi_{v}E^{+}(v)=d\pi_{v}E^{-}(v)=\{w \in T_{\pi(v)}M: w \text{\ is orthogonal to\ } v\}$.
  \item The subspaces $E^{+}(v)$ and $E^{-}(v)$ are continuous and invariant under the geodesic flow.
  \item Let $\tau: SM \rightarrow SM$ be the involution defined by $\tau v=-v$, then
\[E^{+}(-v)=d\tau E^{-}(v)  \text{\ and\ } E^{-}(-v)=d\tau E^{+}(v).\]
  \item If the curvature satisfies $K(p) \geq -a^{2}$ for some $a > 0$,
then $\|K_{v}\xi\| \leq a \|d\pi_{v}\xi\|$ for any $\xi \in E^{+}(v)$ or $\xi \in E^{-}(v)$.
  \item If $\xi \in E^{+}(v)$ or $\xi \in E^{-}(v)$, then $J_{\xi}(t)\neq 0$ for each $t \in \mathbb{R}$.

  \item $\xi \in E^{+}(v)$ (respectively, $\xi \in E^{-}(v)$) if and only if
\[\langle\xi,V(v)\rangle=0  \text{\ and\ }  \|d\pi_{g^{t}v}dg^{t}_{v}\xi\| \leq c\]
for each $t > 0$ (respectively, $t < 0$) and some $c > 0$.
  \item For $\xi \in E^{+}(v)$ (respectively, $\xi \in E^{-}(v)$),
the function $t \mapsto \|J_{\xi}(t)\|$ is non-increasing (respectively, nondecreasing).
\end{enumerate}
\end{proposition}
\vspace{.1cm}

When $\gamma_v(t)$ is a flat geodesic, there exists a non-trivial element $\xi\in E^+(v)\cap E^-(v)$, and $J_\xi$ is a parallel Jacobi field along $\gamma_v(t)$, i.e., $J_\xi'(t)=0, \forall\ t\in \mathbb{R}$. In this case, $E^+(v)$ and $E^-(v)$ do not span the whole second tangent space $T_vSM$. The distributions $E^s$ and $E^u$ on $SM$ are  integrable and their integral manifolds form foliations $W^s$ and $W^u$ of $SM$, respectively. These two foliations are both invariant under $g^t$, known as the stable and unstable horocycle foliations.\\

\subsection{Universal Cover}

Let $\widetilde{M}$ be the universal Riemannian cover of $M$, i.e., a simply connected complete Riemannian manifold for which $M=\widetilde{M}/\Gamma$ where $\Gamma$ is a discrete subgroup of the group of isometries of $\widetilde{M}$, isomorphic to $\pi_1(M)$. Recall that we assume $M$ has no focal points. According to Hadmard-Cartan Theorem, for each two points on $\widetilde{M}$ there is a unique geodesic segment joining them. Therefore $\widetilde{M}$ can be identified with the open unit disk in the plane. The lifting of a geodesic $\gamma$ from $M$ to $\widetilde{M}$ is denoted by $\widetilde{\gamma}$. Two geodesics $\widetilde\gamma_1$ and $\widetilde\gamma_2$ are said to be \emph{asymptotes} if $d(\widetilde\gamma_1(t), \widetilde\gamma_2(t)) \leq C$ for some $C>0$ and $\forall\ t>0$. It is easy to check that the asymptotes relation is an equivalence relation. Let $\widetilde{M}(\infty)$ be the set of all the equivalence classes, which can be identified with the boundary of the unit disk. Then the set $$\overline{M}:=\widetilde{M}\cup\widetilde{M}(\infty)$$ can be identified with the closed unit disk in the plane. Denote by $\widetilde{\gamma}(+\infty)$ the asymptote class of the geodesic $\widetilde{\gamma}$, and by $\widetilde{\gamma}(-\infty)$ the one of the reversed geodesic of $\widetilde{\gamma}$.
We use $\widetilde{W}^s$ and $\widetilde{W}^u$ to denote the lifting of $W^s$ and $W^u$ to $S\widetilde{M}$, respectively. It is obvious that if $w \in \widetilde{W}^s(v)$, then geodesics $\widetilde{\gamma}_v(t)$ and $\widetilde{\gamma}_w(t)$ are asymptotic.\\

An isometry $\alpha$ of $\widetilde{M}$ is called \emph{axial} if there exist a geodesic $\widetilde{\gamma}$ on $\widetilde{M}$ and a $t_{1}>0$ such that for all $t\in \mathbb{R}$, $\alpha(\widetilde\gamma(t)) = \widetilde\gamma(t+t_{1})$. The corresponding geodesic $\widetilde\gamma$ is called an \emph{axis} of $\alpha$. The following result is due to Watkins \cite{Wat}, which is proved for rank $1$ manifolds without focal points. Here we only need it for surfaces without focal points.
\vspace{.1cm}

\begin{lemma}[Cf. Theorem 6.11 in \cite{Wat}]\label{watkins}
Let $\widetilde\gamma$ be an axis of an isometry $\alpha$ of $\widetilde{M}$. Suppose that $\widetilde\gamma$ is not a flat geodesic.
Then for all neighborhoods $U \subseteq \overline{M}$ of $\widetilde\gamma(-\infty)$ and $V \subseteq \overline{M}$ of $\widetilde\gamma(+\infty)$,
there is an integer $N \in \mathbb{N}$ such that
$$\alpha^{n}(\overline{M}-U)\subseteq V, ~~\alpha^{-n}(\overline{M}-V)\subseteq U,$$
for all $n \geq N$.
\end{lemma}
\vspace{.1cm}

Obviously every closed geodesic $\gamma$ in $M$ can be lifted to a geodesic $\widetilde{\gamma}$ on $\widetilde{M}$, such that
$$\widetilde{\gamma}(t+t_0)=\phi(\widetilde{\gamma}(t)), \ \ \forall t \in \mathbb{R},$$
for some $t_0 >0$ and $\phi \in \pi_1(M)$. Therefore, $\widetilde\gamma$ is an axis of $\phi$. In this case, we also say that $\phi$ fixes $\widetilde{\gamma}$ , written as $\phi(\widetilde{\gamma})=\widetilde{\gamma}$. $\phi$ acts on $\widetilde{M}(\infty)$ in a natural way and fixes exactly the two points $\widetilde{\gamma}(\pm \infty)$. Moreover, by Lemma \ref{watkins}, for any $x\in \widetilde{M}(\infty)$, $x \neq \widetilde{\gamma}(\pm \infty)$, we have $$\lim_{n\to +\infty}\phi^n(x)= \widetilde{\gamma}(+\infty) ~~\ ~~\mbox{and} ~~\lim_{n\to -\infty}\phi^n(x)= \widetilde{\gamma}(-\infty).$$\vspace{.1cm}

\section{The regular set}\label{section 3}
This section is devoted to proving Theorem \ref{regular}. The proof of Theorem \ref{regular} in the nonpositive curvature case is given in Lemma 1.1 in \cite{Wu}, though it is already well known in folklore that,  in the non-positive curvature case after Pesin's Theorem \ref{pesin}, all that remains for ergodicity of the geodesic flow is to show that $\Lambda$ has zero Liouville measure. Nevertheless, to prove Theorem \ref{regular} in the no focal points case, we need use some geometric properties of the geodesic flow which we will present below.

For a given $\xi\in T_vSM$, we always let $J_{\xi}(t)$ be the unique Jacobi field satisfying the Jacobi equation \eqref{e:jacobi0} under initial conditions
\[J_{\xi}(0)=d\pi_{v}\xi,\ ~~\frac{d}{dt}\Big|_{t=0}J_{\xi}(t)=K_{v}\xi.\]
Suppose that $J_{\xi}(t)$ is perpendicular to $\gamma'_v(t)$, then $J_\xi(t)=j_\xi(t)e_2(t)$ and $j_\xi(t)=\|J_\xi(t)\|$. Denote $u_\xi(t)=j'_\xi(t)/j_\xi(t)$. Recall that $u_\xi$ is a solution of the Riccati equation \eqref{e:ricatti}.

Given $\xi\in T_vSM$, the \emph{Lyapunov exponent} $\chi(v, \xi)$ is defined as
$$\chi(v, \xi):=\limsup_{T\to \infty}\frac{1}{T}\log\|dg^T\xi\|.$$
The following proposition shows the connection between the Lyapunov exponent $\chi(v, \xi)$ and the function $u_{\xi}$.
\vspace{.1cm}

\begin{proposition}\label{lyapunov}
For any $v\in SM$ and $\xi\in E^+(v)$, one has
\[\chi(v, \xi)=\limsup_{T\to \infty}\frac{1}{T}\int_0^Tu_\xi(t)dt.\]
\end{proposition}
\begin{proof}
By the definition of Lyapunov exponents and Proposition \ref{subspace}(5), we have
\begin{equation*}
\begin{aligned}
\chi(v, \xi)&=\limsup_{T\to \infty}\frac{1}{T}\log\|dg^T\xi\|=\limsup_{T\to \infty}\frac{1}{T}\log \sqrt{\|J_\xi(T)\|^2+\|J'_\xi(T)\|^2}\\
&=\limsup_{T\to \infty}\frac{1}{T}\log \|J_\xi(T)\|=\limsup_{T\to \infty}\frac{1}{T}\int_0^T(\log j_\xi(t))'dt\\
&=\limsup_{T\to \infty}\frac{1}{T}\int_0^T\frac{j_\xi'(t)}{j_\xi(t)}dt=\limsup_{T\to \infty}\frac{1}{T}\int_0^Tu_\xi(t)dt.
\end{aligned}
\end{equation*}
\end{proof}
\vspace{-.1cm}
Throughout this section, if $\xi\in E^+(v)$, we write $j(t):=j_\xi(t)$ and $u(t):=u_\xi(t)$ for simplicity. By Proposition \ref{subspace}(8) and the definition of $u(t)$, we know that $u(t)\leq 0, \ \forall t\in\mathbb{R}$.\\

The following notion of uniformly recurrent vectors appeared in \cite{BBE}:
\vspace{.1cm}

\begin{definition}[Cf. \cite{BBE}]\label{uniformlyrecurrent}
A vector $x\in SM$ is said to be uniformly recurrent if for any neighborhood $U$ of $x$ in $SM$
$$\liminf_{t \to \infty} \frac{1}{T}\int_0^T \mathbf{I}_U(g^t(x))dt > 0,$$
where $\mathbf{I}_U$ is the characteristic function of $U$.
\end{definition}
\vspace{.1cm}

The next lemma about the set of uniformly recurrent vectors was stated in \cite{BBE} without a proof, so we provide a proof here. It will be used later in our proof of Theorem \ref{regular}.
\vspace{.1cm}

\begin{lemma}\label{uniformrec}
Let $\Gamma$ be the set of all the uniformly recurrent vectors. Then $\Gamma$ has full Liouville measure.
\end{lemma}

\begin{proof}
Let $\{U_n\}_{n\in\mathbb{N}}$ be a countable base consisting of open sets on $SM$. By Birkhoff ergodic theorem, there exists a set $X \subset SM$ of full measure such that for all $x \in X$ and all $n\in\mathbb{N}$, the limit
$$f_n(x):=\lim_{T \to \infty} \frac{1}{T}\int_0^T \mathbf{I}_{U_n}(g^t(x))dt$$
exists. And
$$\int_{SM}f_n(x)d\nu(x)= \nu(U_n).$$
\vspace{.1cm}

Assume the contrary that $\nu(\Gamma^c)>0$. Then $\Gamma^c \cap X$ is non-empty. For each $y \in\Gamma^c \cap X$, which is not uniformly recurrent, there exists a neighborhood $U$ of $y$ in $SM$ such that
\begin{equation*}\label{e:not}
 \liminf_{T \to \infty} \frac{1}{T}\int_0^T \mathbf{I}_U(g^t(y))dt =0.
\end{equation*}
Then there exists an $n(y)$ such that $U_{n(y)} \subset U$, and
\begin{equation}\label{e:f}
 f_{n(y)}(y)=\lim_{T \to \infty} \frac{1}{T}\int_0^T \mathbf{I}_{U_n}(g^t(y))dt \leq \liminf_{T \to \infty} \frac{1}{T}\int_0^T \mathbf{I}_U(g^t(y))dt =0.
\end{equation}
Since there are only countably many $U_n$, we can find some $N$ such that $\nu (U_N \cap \Gamma^c \cap X)>0$. By \eqref{e:f},  $f_N(y)=0$ for any $y \in U_N \cap \Gamma^c\cap X$.\\

On the other hand, Birkhoff ergodic theorem implies that for a.e. $y \in U_N \cap \Gamma^c\cap X$, one has
$$g(y):=\lim_{T \to \infty} \frac{1}{T}\int_0^T \mathbf{I}_{(U_N \cap \Gamma^c\cap X)}(g^t(y))dt$$
exists with
\begin{equation}\label{e:gamma}
\int_{SM}g(y)d\nu(y)=\nu(U_N \cap \Gamma^c\cap X) >0.
\end{equation}
However, by \eqref{e:f}, we have $g(y) \leq f_N(y)=0$ for all $y \in U_N \cap \Gamma^c\cap X$, which contradicts to \eqref{e:gamma}. This proves the lemma.
\end{proof}
\vspace{.2cm}

Given an open set $U\subset SM$ and a unit vector $w\in SM$, we say that the orbit $g^tw$ has positive frequency of return to $U$ if $\liminf_{t\to \infty}\frac{T_U(t)}{t} >0$, where $T_U(t)$ denotes the total length of the set
\[\mathfrak{T}_U(t):=\{\tau: 0\leq \tau \leq t \text{\ and\ }g^\tau w \in U\}.\]
Now we are ready to prove Theorem \ref{regular}.
\begin{proof}[Proof of Theorem \ref{regular}]
Choose an arbitrary $v\in \Lambda^c \cap \Delta^c \cap \Gamma$, where $\Gamma$ denotes the set of uniformly recurrent vectors. Recall that $\nu(\Gamma)=1$ by Lemma \ref{uniformrec}. Without loss of generality, we assume $v\in (\Delta^+)^c$.  We claim that $K(\gamma_v(t))\geq 0,\ \forall t \in \mathbb{R}$.\\

Assume the contrary that $K(\gamma_v(t_0))<0$ for some $t_0>0$. Since $K(\gamma_v(t_0))<0$, we can choose two open neighborhoods $W_1 \supset W_2$ of $g^{t_0}v$, such that $-\delta_2< K|_{\pi(W_1)}<-\delta_1<0$ and $\text{dist}(\partial W_1, \partial W_2)>\sigma$, for some $\delta_2>\delta_1>0$ and $\sigma>0$.\\

Choose an open neighborhood $U$ of $v$ which is small enough, such that for any $w\in U$ one has $g^{t_0}(w)\in W_2.$ Since $\liminf_{T\to \infty}\frac{1}{T}\int_0^T\mathbf{I}_U(g^tv)dt >0$, we have
\[\liminf_{T\to \infty}\frac{1}{T}\int_{t_0}^{T+t_0}\mathbf{I}_{W_2}(g^tv)dt >0.\]
Then the orbit of $v$ has positive frequency of return to $W_2$, that is
\begin{equation}\label{e:frequency}
\begin{aligned}
\liminf_{T\to \infty}\frac{T_{W_2}(T)}{T}&=\liminf_{T\to \infty}\frac{1}{T+t_0}\int_{0}^{T+t_0}\mathbf{I}_{W_2}(g^tv)dt\\
&=\liminf_{T\to \infty}\frac{T}{T+t_0}\cdot\frac{1}{T}\int_{t_0}^{T+t_0}\mathbf{I}_{W_2}(g^tv)dt>0.
\end{aligned}
\end{equation}
\vspace{-.1cm}

\begin{lemma}\label{negative}
There exists a constant $c>0$ such that if $g^tv\in W_2$, then $u(t)\leq -c$, for all $t\geq 0$.
\end{lemma}
\begin{proof}[Proof of Lemma \ref{negative}]
We prove the lemma by contradiction. Assume this is not true. Then there exists a sequence of $t_i \geq 0$ with $g^{t_i}v\in W_2$, but $u(t_i)\to 0$ as $i\to \infty$.
There exist $s_{i,1}<t_i<s_{i,2}$ such that $g^{s_{i,1}}v, g^{s_{i,2}}v\in \partial W_1$ and $g^tv\in W_1$ for any $s_{i,1}<t<s_{i,2}$.
In fact, $s_{i,1}+\sigma < t_i < s_{i,2}-\sigma$.\\

Recall the Riccati equation $u'(t)+u^2(t)+K(t)=0$. If $i$ is large enough, then $u'(t_i)=-u^2(t_i)-K(t_i)>\delta_1>0$.
We claim that $u(t)$ is strictly increasing in the interval $(t_i, s_{i,2})$.
Indeed, if not, there is a smallest number $s_i\in (t_i, s_{i,2})$ such that $u'(s_i)=0$.
Then $u(s_i)>u(t_i)$, since $u'(t)>0$ for all $t\in (t_i, s_i)$. Therefore $u'(s_i)=-u^2(s_i)-K(s_i)>\delta_1>0$, which is a contradiction.\\

It follows that $u'(t)=-u^2(t)-K(t)>\delta_1>0$ for all $t\in (t_i, s_{i,2})$. Thus
\begin{equation*}
\begin{aligned}
u(s_{i,2})&=u(t_i)+\int_{t_i}^{s_{i,2}}u'(t)dt\\
&>u(t_i)+\delta_1(s_{i,2}-t_i)>u(t_i)+\delta_1\sigma.
\end{aligned}
\end{equation*}
If $i$ is large enough, then $u(t_i)$ is close enough to $0$, and hence $u(s_{i,2})>0$. This contradicts to the fact that $u(t)\leq 0, \forall t\in \mathbb{R}$.
\end{proof}
\vspace{.1cm}

Let us go on with the proof of Theorem \ref{regular}. By Proposition \ref{lyapunov}, Lemma \ref{negative} and \eqref{e:frequency}, one has \begin{equation*}
\begin{aligned}
\chi(v, \xi)=\limsup_{T\to \infty}\frac{1}{T}\int_0^Tu(t)dt\leq \limsup_{T\to \infty}\frac{1}{T}\cdot T_{W_2}(v)\cdot (-c)<0
\end{aligned}
\end{equation*}
where $\xi\in E^+(v)$. This contradicts to $v\in (\Delta^+)^c$. Thus $K(\gamma_v(t))\geq 0$ for all $t\geq 0$.
Analogously, we can prove that $K(\gamma_v(t))\geq 0$ for all $t\leq 0$. Thus $$K(\gamma_v(t))\geq 0, \forall\ t\in \mathbb{R}.$$
\vspace{-.2cm}

Now recall the Riccati equation $u'(t)+u^2(t)+K(t)=0$ again. Since $K(t)\geq 0$ along $\gamma_v(t)$, we have $u'(t)\leq 0$ for all $t\in \mathbb{R}$. We have the following three possibilities:
\begin{enumerate}
  \item $\lim_{t\to \infty}u(t)=0$. Since $u(t)\leq 0$ and $u'(t)\leq 0$ for all $t\in \mathbb{R}$, we must have $u(t)\equiv 0, \forall t\in \mathbb{R}$. Then $u'(t)\equiv 0, \forall t\in \mathbb{R}$. It follows from the Riccati equation that $K(t)\equiv 0, \forall t\in \mathbb{R}$.
  \item $\lim_{t\to \infty}u(t)=-d<0$ for some $d>0$. Then $$u'(t)=-u^2(t)-K(t)\leq -u^2(t)<0, \ \forall t\in \mathbb{R}.$$ This contradicts to the fact that $\lim_{t\to \infty}u'(t)=0$.
  \item $\lim_{t\to \infty}u(t)=-\infty$. Since $J(t)$ is a stable Jacobi field along $\gamma_v(t)$, we have $|u(t)|=|j'(t)/j(t)|\leq a$. We also arrive at a contradiction.
\end{enumerate}
In summary, we must have $K(\gamma_v(t))\equiv 0,\ \forall t\in \mathbb{R}$. This contradiction to our assumption $v\in\Lambda^c$. Therefore $\Lambda^c\cap(\Delta^+)^c\cap\Gamma=\emptyset$. The case $v\in (\Delta^-)^c$ can be delt with similarly and leads to the same result. Based on the discussion in the above, we can conclude that $\Lambda^c\cap\Delta^c\cap\Gamma=\emptyset$. Since $\Gamma$ is a full measure set, we immediately know that $$\nu(\Lambda^c\cap\Delta^c)=0.$$ We are done with the proof of Theorem \ref{regular}.
\end{proof}
\vspace{.5cm}

\section{Flat geodesics}\label{section 4}

\subsection{Flat strips are closed}

A flat strip means a totally geodesic isometric imbedding $r:\mathbb{R}\times [0,c]\rightarrow \widetilde{M}$, where $\mathbb{R}\times [0,c]$ is a strip in a Euclidean plane. The projection of a flat strip from $\widetilde M$ to $M$ is also called a flat strip. We have the following flat strip lemma:
\vspace{.1cm}

\begin{lemma}[Cf. \cite{OS}] \label{flat strip lemma}
If two distinct geodesics $\widetilde{\alpha}$ and $\widetilde{\beta}$ satisfy $d(\widetilde{\alpha}(t), \widetilde{\beta}(t))\leq C$ for some $C>0$ and $\forall t\in \mathbb{R}$, then they are the boundary curves of a flat strip in $\widetilde{M}$.
\end{lemma}
\vspace{.1cm}

The flat strip lemma for nonpositively curved manifolds was established by Eberlein and O'Neill in \cite{EO}. The above flat strip lemma for manifolds without focal points is due to Green in dimension two (cf.~\cite{Gre}), and O'Sullivan (cf.~\cite{OS}) in arbitrary dimensions. The following lemma is also useful in our work.
\vspace{.1cm}

\begin{lemma}[Cf. Lemma 3.6 in \cite{Wu}] \label{boundary}
If $w' \in W^s(w)\subset SM$ and
$$\lim _{t\to +\infty} d(\gamma_w(t), \gamma_{w'}(t))=\delta >0,$$
then $\gamma_w(t)$ and $\gamma_{w'}(t)$ converge to the boundaries of a flat strip of width $\delta$.
\end{lemma}
\vspace{.1cm}

In the view of Conjecture \ref{conjecture}, our aim is to show that all flat geodesics are closed. An important progress was made by Cao and Xavier on the flat geodesics inside flat strips on manifolds of nonpositive curvature, in an unpublished preprint \cite{CX}. We state it in the following theorem.
\vspace{.1cm}

\begin{theorem}[Cao-Xavier, cf.~\cite{CX}]
Let $M$ be a smooth, connected and closed surface with genus $\mathfrak{g} \geq 2$. Suppose that $M$ has nonpositive curvature. Then any flat strip on $M$ consists of closed geodesics in the same homotopy type.
\end{theorem}
\vspace{.1cm}

Based on the flat strip lemma \ref{flat strip lemma}, we generalize the above result to the manifolds without focal points. We adapt the argument of Cao-Xavier to surfaces without focal points.
\vspace{.1cm}

\begin{theorem} \label{flat strip closed}
Let $M$ be a smooth, connected and closed surface with genus $\mathfrak{g} \geq 2$. Suppose that $M$ has no focal points. Then any flat strip on $M$ consists of closed geodesics in the same homotopy type.
\end{theorem}
\begin{proof}
Observe that in the universal cover $\widetilde{M}$, there exists an upper bound for the width of all the flat strips. Indeed, let $D > \text{diam}(M)$. Then a flat strip of width greater than $2D$ contains a fundamental domain in $\widetilde{M}$. Hence $M$ must be a flat torus. This contradicts to the fact that $M$ has genus $\mathfrak{g} \geq 2$.\\

Let $\widetilde{G}: (-\infty, \infty)\times [0, \epsilon_0]\to \widetilde{M}$ be a flat strip in $\widetilde{M}$, and $G=p(\widetilde{G})$ where $p: \widetilde{M}\to M$ is the universal covering map. Consider a sequence of unit vectors $v_i\in SM$ where $v_i=\frac{\partial G}{\partial t}(i, \epsilon_0),\ i=1,2, \cdots$. Since $SM$ is compact, there exists a subsequence of $\{v_i\}_{i=1}^\infty$ which converges to a unit vector $v_0\in SM$. For simplicity of notations, we still let $\{v_i\}_{i=1}^\infty$ denote the subsequence. Recall that $\pi: SM \to M$ is the canonical projection. Let $x_i$ denote the foot point of $v_i$, i.e. $x_i=\pi(v_i),\ i=0,1,2, \cdots$.\\

Let $\delta$ be the injectivity radius of $M$. For sufficiently large $j$, we may assume $d(v_j, v_0)<\delta/2$. We choose a preimage $\widetilde{x}_0\in p^{-1}(x_0)$ such that $\widetilde{x}_0$ is the nearest point to the flat strip $\widetilde{G}$ in $p^{-1}(x_0)$. Then $p|_{B(\widetilde{x}_0, \delta)}: B(\widetilde{x}_0, \delta) \to B(x_0, \delta)$ and $\Psi:=dp|_{SB(\widetilde{x}_0, \delta)}: SB(\widetilde{x}_0, \delta) \to SB(x_0, \delta)$ are both isometries.\\

Denote $w_i=\Psi^{-1}(v_i)\in SB(\widetilde{x}_0, \delta),\ i=0,1,2, \cdots$. Let $F_j: (-\infty, \infty)\times [0, \epsilon_0]\to \widetilde{M}$ denote the lifted flat strip tangent to $w_j,\ j=1,2, \cdots$. Then the limit of $F_j$ is a flat strip $\widetilde{G}_0: (-\infty, \infty)\times [0, \epsilon_0]\to \widetilde{M}$ tangent to $w_0$. There are two distinct cases.
\begin{enumerate}
  \item \textbf{$p\circ F_{j_0}$ is periodic for some $j_0\in \mathbb{N}$.}
  As $p\circ F_{j_0}$ is the flat strip tangent to $v_{j_0}$, it coincides with $G$. Hence $G$ is periodic.

  \item \textbf{$p\circ F_{j}$ is not periodic for any $j\in \mathbb{N}$.}
  Then $F_j$ and $\widetilde{G}_0$ are a pair of transversal flat strips of the same width $\epsilon_0$. Suppose that they intersect at $q_j$ with angle $\alpha_j$, where $q_j\in \partial F_j\cap \widetilde{G}_0((-\infty, \infty)\times \{\epsilon_0\})$, $j=1, 2, \cdots$. Because $F_j\cup \widetilde{G}_0$ has curvature $0$ everywhere, we can construct a rectangle $R_j=[0, L_j]\times (0, \epsilon_0/8]$ contained in the closure of $F_j-\widetilde{G}_0$ such that
  \begin{itemize}
    \item one side of $R_j$, $[0, L_j]\times \{0\}$ is contained in the line $\widetilde{G}_0((-\infty, \infty)\times \{\epsilon_0\})$,
    \item $L_j \geq \frac{\epsilon_0}{16\sin \alpha_j}$.
  \end{itemize}
  Attaching $R_j$ to $\widetilde{G}_0$, we obtain an isometric embedding $\widetilde{G}_0: [c_j, c_j+L_j]\times [0, \frac{9\epsilon_0}{8}]$ for some $c_j\in \mathbb{R}$. Let $\widetilde{u}_j$ be the unit vector tangent to $\widetilde{G}_0([c_j, c_j+L_j]\times \{0\})$ at the point $\widetilde{G}_0(c_j+L_j/2,0)$. Write $u_j=dp(\widetilde{u}_j)$ and suppose that a subsequence of $\{u_j\}$ converges to $u_0$. As $L_j \to \infty$ as $j\to \infty$, we know that there exists a flat strip tangent to $u_0$ of width $\frac{9\epsilon_0}{8}$. Hence there exits a flat strip $\widetilde{G}_1$ of width $\frac{9\epsilon_0}{8}$ in $\widetilde{M}$.
\end{enumerate}
\vspace{.1cm}

We are done if we arrive at the first case above. If we have the second case, we then repeat the argument for the new flat strip $\widetilde{G}_1$. But we cannot enlarge our flat strips by a factor $9/8$ again and again, as the width of the flat strips in $\widetilde{M}$ have an upper bound $2D$. Thus we must arrive at the first case at some step. It follows that $G$ is periodic.
\end{proof}
\vspace{.2cm}

\subsection{Expansivity}
The proof of Theorems \ref{A} and \ref{B} uses an argument based on the following \emph{expansivity} property of a vector $x\in SM$ not tangent to a flat strip. This argument will be used for several times later, in the proof of Theorems \ref{A}, \ref{B}, and Lemma \ref{flat closed geodesic}.
\vspace{.1cm}

\begin{definition}[Cf. Definition 3.2.11 in \cite{KH}]
We say $x\in SM$ has the expansivity property if there exists a small $\delta_0>0$, such that whenever $d(g^t(x), g^t(y))< \delta_0$, $\forall t\in \mathbb{R}$, then $y=g^{t_0}(x)$ for some $t_0$ with $|t_0|<\delta_0$.
\end{definition}
\vspace{.1cm}

The flat strip lemma (Lemma \ref{flat strip lemma}) for surfaces without focal points guarantees the expansivity property for a unit vector which is not tangent to a flat strip.
\vspace{.1cm}

\begin{lemma}\label{expansivity}
If $x\in SM$ is not tangent to a flat strip, then it has the expansivity property.
\end{lemma}
\begin{proof}
We prove this lemma by contradiction. Assume the lemma does not hold. Then for an arbitrarily small $\epsilon >0$ less than the injectivity radius of $M$, there exists a point $y\in SM$ such that $y\notin \mathcal{O} (x)$ and $d(\gamma_x(t), \gamma_y(t))< \epsilon$, $\forall t\in \mathbb{R}$. By the choice of $\epsilon$, we can lift $\gamma_x(t)$ and $\gamma_y(t)$ to the universal cover $\widetilde{M}$ such that
$$d(\widetilde{\gamma}_x(t), \widetilde{\gamma}_y(t))< \epsilon,\ \ \forall t\in \mathbb{R}.$$
Thus by the flat strip lemma \ref{flat strip lemma}, $\widetilde{\gamma}_x(t)$ and $\widetilde{\gamma}_y(t)$ bound a flat strip. Hence $x$ is tangent to a flat strip, a contradiction.
\end{proof}
\vspace{.3cm}

\subsection{Area of ideal triangles}
Given $x,y,z \in \widetilde{M}(\infty)$, an ideal triangle with vertices $x, y, z$ means the region in $\widetilde{M}$ bounded by the three geodesics joining the vertices. In the case when at least one of $x, y, z$ is on $\widetilde{M}(\infty)$ (the other points can be inside $\widetilde{M}$), we also call the region bounded by the three geodesics an ideal triangle. The following theorem about the ideal triangle is proved in \cite{RR}.
\vspace{.1cm}

\begin{theorem}[Cf. Theorem 1 in \cite{RR}]\label{model}
Let $M$ be a $C^3$ compact surface having negative curvature at all points, except along a simple closed geodesic $\gamma(t)$ where the curvature is zero for every $t\in \mathbb{R}$. Let $N$ be a normal neighborhood of $\gamma$ and let $y: N\to \mathbb{R}$ be the distance $y(p)$ from a point $p\in N$ to $\gamma$. Suppose that on $N$ the Gaussian curvature $K$ satisfies
$$K(p)=-y(p)^\epsilon f(y(p))$$
where $\epsilon \geq 2$ and $f(y)$ is an analytic function with $f(0)=1$. Then the area of ideal triangle having $\gamma$ as an edge is infinite.
\end{theorem}
\vspace{.1cm}

Roughly speaking, the above theorem follows from the fact that the length of stable Jacobi fields decreases in a sub-exponential rate along a geodesic with curvature close to zero (cf. \cite{RR}). Combining this theorem with the classical comparison theorems, Ruggiero also proved that for nonpositively curved surfaces the area of an ideal triangle having a flat geodesic as an edge is infinite (Lemma 4.1 in \cite{RR}).\\

The following theorem is a version of the above result for surfaces without focal points. For completeness, we provide a sketch of the proof.
\vspace{.1cm}

\begin{theorem}\label{ideal triangle}
Let $M$ be a smooth, connected and closed surface with genus $\mathfrak{g} \geq 2$. Suppose that $M$ has no focal points. If $K(\widetilde{\gamma}(t))\equiv 0$, for $\forall t\in \mathbb{R}$, then every ideal triangle having $\widetilde{\gamma}(t)$ as an edge has infinite area.
\end{theorem}
\begin{proof}
First, we construct a model surface. The comparison theory will be applied to this model surface and the surface without focal points that we are considering in the theorem. Consider the annulus $A=S^1\times \mathbb{R}$ endowed with the metric $g_a(s,t)=Y^2(t)ds^2+dt^2$, where
$Y(t)=1+at^4$ for some constant $a>0$. Obviously, for a fixed $s$, $\gamma_s(t)=(s,t)$ is a geodesic with the arclength parameter $t$. The circle $\gamma^0(s)=(s,0)$ is also a geodesic on $A$. By a simple computation, we know the curvature of $A$ is
\[K_{g_a}(s,t)=\frac{-Y''(t)}{Y(t)}=\frac{-12at^2}{1+at^4}.\]
Since $|t|=|t(p)|$ is the distance from $p=(s,t)\in A$ to the geodesic $\gamma_0$, and $K(s,t)=0$ if and only if $t = 0$, the geodesic $\gamma^0$ has a tubular neighborhood $N$ satisfying the hypothesis of Theorem \ref{model}.\\

Now we consider the surface $(M,g)$ in the theorem, that is, a smooth, connected and closed surface of genus $\mathfrak{g} \geq 2$, which has no focal points. Consider a flat geodesic $\gamma(s)$ on $M$. Let $\widetilde\gamma(s)$ be an arbitrary lifting of $\gamma(s)$ on $\widetilde{M}$ and consider a normal tubular neighborhood $N_\epsilon(\widetilde\gamma)$ on $\widetilde{M}$. One can consider the Fermi coordinates in $N_\epsilon(\widetilde\gamma)$, that is $$F: \mathbb{R}\times (-\epsilon, \epsilon)\to N_\epsilon(\widetilde\gamma),\ \ F(s,t)=\exp_{\widetilde\gamma(s)}(tn(s)),$$ where $n(s)$ is a smooth unit vector field everywhere normal to $\widetilde\gamma'(s)$. Let $A_g=(\mathbb{R}\times (-\epsilon, \epsilon), F^*g)$ be the strip $\mathbb{R}\times (-\epsilon, \epsilon)$ endowed with the pullback of the metric $g$ by $F$. By the Taylor's formula and the
compactness of $M$, there exist $\epsilon_1\leq \epsilon$ and $C>0$, such that for every $s\in \mathbb{R}$ the curvature of $A_g$ can be written as
\[K_g(p)=-y(p)^2f_s(y(p)),\]
where $f_s: (-\epsilon_1, \epsilon_1)\to \mathbb{R}$ is a $C^\infty$ function with $f_s(t)\leq C$ for every $t\in (-\epsilon_1, \epsilon_1)$. Since $t=t(p)$ is the distance from $p$ to $F(s,0)=\widetilde\gamma(s)$, we have
\[K_g(s,t)\geq -t^2C,\]
where $K(s,t)$ is the curvature of $A_g$. Pick an $a>\frac{C}{12}$. Then
\begin{equation*}
\begin{aligned}
K_{g_a}(s,t)=\frac{-12at^2}{1+at^4}=-12at^2(1-at^4+O(t^8))\leq -Ct^2
\end{aligned}
\end{equation*}
for every $(s,t)\in \mathbb{R}\times (-\delta_a, \delta_a)$, where $\delta_a>0$ depends on $a$. So we know that
\[K_{g}(s,t)\geq -t^2C \geq K_{g_a}(s,t),\]
for every $(s,t)\in \mathbb{R}\times (-\delta_a, \delta_a)$.\\

The following lemma follows from a standard comparison argument. See \cite{CE} for more details on the comparison theory.
\vspace{.1cm}

\begin{lemma}\label{comparison}
Let $a>0, \delta_a>0$, and $A_g, A_{g_a}$ be the strips defined as above. Suppose that the horizontal line $\widetilde\gamma(s)=(s,0)$ is a geodesic for both $A_g$ and $A_{g_a}$. Let $p,q\in \mathbb{R}\times (-\delta, \delta)$, and $[p, q]_g, [p, q]_{g_a}$ be the geodesic segments joining $p$ to $q$ under the metrics $g$ and $g_a$ respectively. Let $\gamma_{s_0}(t)=(s_0, t)$. Then
\[d_g([p, q]_g\cap \gamma_{s_0}, \widetilde\gamma(s_0)) > d_{g_a}([p, q]_{g_a}\cap \gamma_{s_0}, \widetilde\gamma(s_0)),\]
where $d_g$ and $d_{g_a}$ are the distances with respect to the metric $g$ and $g_a$ respectively.
\end{lemma}
\vspace{.1cm}

Now let $p\in A_g$ and let $\widetilde\gamma_{g,p}, \widetilde\gamma_{g_a,p}$ be asymptotic to $\widetilde\gamma$ starting at $p$. Consider the ideal
triangles
\begin{equation*}
\begin{aligned}
\Delta&=[\widetilde\gamma(0),p]_g\cup \widetilde\gamma_{g,p}[0,+\infty)\cup \widetilde\gamma[0,+\infty)\\
\Delta_{g_a}&=[\widetilde\gamma(0),p]_g\cup \widetilde\gamma_{g_a,p}[0,+\infty)\cup \widetilde\gamma[0,+\infty).
\end{aligned}
\end{equation*}
By Lemma \ref{comparison} and a limit argument, we can deduce that $\Delta_{g_a}\subset \Delta_{g}$. So $\text{area}_g(\Delta_{g_a})\leq\text{area}_g(\Delta_{g})$. Then combining Rauch comparison theorems we obtain that
\[\text{area}_{g_a}(\Delta_{g_a})\leq\text{area}_g(\Delta_{g_a})\leq\text{area}_g(\Delta_{g}).\]
By Theorem \ref{model}, $\text{area}_{g_a}(\Delta_{g_a})$ is infinite. Thus $\text{area}_g(\Delta_{g})$ is infinite. This proves the theorem.
\end{proof}
\begin{remark}
A difference between the above proof and the nonpositive curvature case is that we do not require $f_s(t)$ to be a positive function. For our purpose to apply comparison theorems, only the upper bound of $f_s(t)$ is used. Thus, the proof goes almost the same as that of Lemma 4.1 in \cite{RR}.
\end{remark}
\vspace{.3cm}

\subsection{Non-closed flat geodesics}
In this subsection, we discuss some important properties of the flat geodesics. Our Theorem \ref{nonclosedset} is a straightforward corollary of these properties. In fact, it is closely related to the following two lemmas (Lemma \ref{flat closed geodesic0} and \ref{flat closed geodesic}).
The first lemma shows that if a flat geodesic converges to a closed one (no matter flat or not), then the former geodesic must also be closed, and coincide with the latter.
\vspace{.1cm}

\begin{lemma}\label{flat closed geodesic0}
Suppose that $y\in \Lambda$, and the $\omega$-limit set $\omega(y)= \mathcal{O}(z)$ where $\mathcal{O}(z)$ is periodic. Then $\mathcal{O}(y)=\mathcal{O}(z)$. In particular, $\mathcal{O}(y)$ is periodic.
\end{lemma}

\begin{proof}

\begin{figure}[h]
\centering
\setlength{\unitlength}{\textwidth}

\begin{tikzpicture}
\draw (0,0) circle (2);

\draw [name path=line 1](0,2) to [out=260,in=60] (-1,-1.73)node[below] {$\widetilde{\gamma}_0$};
\draw  [name path=line 2](0,2)to [out=270,in=130] (1, -1.73)node[below] {$\widetilde{\gamma}$};
\draw  [name path=line 3](0,2)to [out=280,in=135] (1.42, -1.42)node[right] {$\phi(\widetilde{\gamma})$};

\draw [name path=line 4](-1.73, -1) to [out=355,in=182] (1.73,-1)node[right] {$\widetilde{\alpha}$};
\draw ([name path=line 5]-2,0) to [out=355,in=182] (2,0) node[right] {$\phi(\widetilde{\alpha})$};
\coordinate [label=below:$A$] (A) at (-0.67,-1.05);
\coordinate [label=below:$B$] (B) at (0.51,-1.05);
\coordinate [label=below left:$C$] (C) at (-0.35,-0.05);
\coordinate [label=below left:$D$] (D) at (0.15,-0.05);
\coordinate [label=below right:$E$] (E) at (0.5,-0.03);
\coordinate [label=above:$F$] (F) at (0,2);
\fill (A) circle (1.5pt)(B) circle (1.5pt)(C) circle (1.5pt)(D) circle (1.5pt)(E) circle (1.5pt)(F) circle (1.5pt);
\end{tikzpicture}

\caption[]{Proof of Lemma \ref{flat closed geodesic0}}

\end{figure}
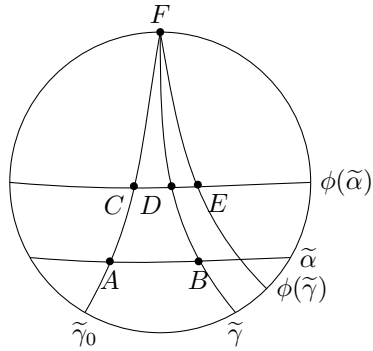

Since $\omega(y)= \mathcal{O}(z)$, we can lift geodesics $\gamma_z(t), \gamma_y(t)$ to the universal cover $\widetilde{M}$, denoted by $\widetilde{\gamma}_0(t)$ and $\widetilde{\gamma}(t)$ respectively, such that $\lim_{t\to +\infty}d(\widetilde{\gamma}_0(t), \widetilde{\gamma}(t))=0$. In particular, $\widetilde{\gamma}_0(+\infty)=\widetilde{\gamma}(+\infty)$.\\

Since $\gamma_z(t)$ is a closed geodesic, there exists an isometry $\phi$ of $\widetilde{M}$ such that $\phi(\widetilde{\gamma}_0(t))=\widetilde{\gamma}_0(t+t_0)$. Moreover, on the boundary of the disk $\widetilde{M}(\infty)$, $\phi$ fixes exactly two points $\widetilde{\gamma}_0(\pm \infty)$, and for any other point $a\in \widetilde{M}(\infty)$, $\lim_{n\to +\infty}\phi^n(a)= \widetilde{\gamma}_0(+\infty)$.\\

Assume $\widetilde{\gamma}$ is not fixed by $\phi$. Then $\widetilde{\gamma}$ and $\phi(\widetilde{\gamma})$ do not intersect, since $\phi(\widetilde{\gamma})(+\infty)=\widetilde{\gamma}(+\infty)$. By Lemma \ref{watkins}, replacing $\phi$ by $\phi^N$ for a large enough $N\in \mathbb{N}$ if necessary, we know that the position of $\phi(\widetilde\gamma)$ must be as shown in Figure 1. We then pick another geodesic $\widetilde{\alpha}$ as in Figure 1. The image of infinite triangle $ABF$ under $\phi$ is the infinite triangle $CEF$. Since $\phi$ is an isometry, it preserves area. With a limit process, it is easy to show that the area of $ABCD$ is no less than the area of $DEF$. But since $\gamma$ is a flat geodesic, the area of $DEF$ is infinite by Theorem \ref{ideal triangle}. We arrive at a contradiction to the fact that $ABCD$ has finite area. So $\phi(\widetilde{\gamma})$ and $\widetilde{\gamma}$ must coincide.\\

Therefore $\widetilde{\gamma}(\pm\infty)= \widetilde{\gamma}_0(\pm\infty)$. Then either $\widetilde{\gamma}(t)$ and $\widetilde{\gamma}_0(t)$ bound a flat strip by the flat strip lemma \ref{flat strip lemma} or $\widetilde{\gamma}(t)=\widetilde{\gamma}_0(t)$. Recall that $\lim_{t\to +\infty}d(\widetilde{\gamma}(t), \widetilde{\gamma}_0(t))=0$, we must have $\widetilde{\gamma}(t)=\widetilde{\gamma}_0(t)$. Hence $\mathcal{O}(y)=\mathcal{O}(z)$.
\end{proof}
\vspace{.1cm}

Lemma \ref{flat closed geodesic0} can be strengthened to the following.
\begin{lemma}\label{flat closed geodesic}
 Suppose that $y\in \Lambda$ and $z\in \omega(y)$ where $z$ is periodic. Then $\mathcal{O}(y)=\mathcal{O}(z)$. In particular, $y$ is periodic.
\end{lemma}
\vspace{.1cm}

The proof of Lemma \ref{flat closed geodesic} follows from an argument similar to the one in the proof of Lemma 3.8 in \cite{Wu}. The argument relies on the expansivity property of a unit vector not tangent to a flat strip, as stated in Lemma \ref{expansivity}. Then Theorem \ref{nonclosedset} follows from an almost same argument. For this reason, we omit the proof here. Readers can check the proof of Lemma 3.8 and Theorem 1.5 in \cite{Wu} for details of the argument.\\

Now, we consider the non-closed flat geodesics. The lemma in the following describes an important property of the non-closed flat geodesics.
\vspace{.1cm}

\begin{lemma}\label{selfintlem}
A non-closed geodesic on $M$ along which the curvature is everywhere $0$ must intersect itself.
\end{lemma}
\begin{proof}
Let $\gamma$ be a non-closed geodesic on $M$ along which the curvature is everywhere $0$. Assume that $\gamma$ has no self-intersections. Then the lifts of $\gamma$ to the universal cover $\widetilde{M}$ are pairwise disjoint and can be permuted by deck transformations. Furthermore, the identity is the only deck transformation that takes a lift into itself.\\

Let $R$ be a component of the complement of the lifts in the universal cover and $\widetilde \gamma$ be one of the lifts that bounds $R$. Pick a point $\widetilde p$ in $R$ that is close to $\widetilde \gamma$ and let $\widetilde \beta$ be the geodesic that starts from $p$ and is asymptotic to $\widetilde \gamma$. Then $\widetilde \beta$ and $\widetilde \gamma$ are two sides of an ideal triangle $T$  that lies in $R$. Suppose that $\widetilde p$ is close enough to $\widetilde \gamma$ such that the only points of $T$ that lie on a lift of $\gamma$ are those on $\widetilde \gamma$.\\

Since the curvature vanishes everywhere along $\widetilde \gamma$, the area of $T$ is infinite. Hence there is a point $q\in M$ that is covered by an infinite sequence $\widetilde q_0,  \widetilde q_1, \widetilde q_2,\dots$ of points in $T$. These points all belong to the same orbit of the fundamental group (the orbit that covers $q$). Since the action of the fundamental group is properly discontinuous, the sequence $\{\widetilde q_i \}$ must converge to the vertex of $T$ at $\widetilde\gamma(\infty)$. Therefore there is a sequence $t_i \to \infty$ such that $d(\widetilde\gamma(t_i),\widetilde q_i) \to 0$.\\

Let $\psi_i$ be the deck transformation that maps $\widetilde q_i$ to $\widetilde q_0$. Then $\psi_i(\widetilde\gamma(t_i))$ converges to $\widetilde q_0$ as $i \to \infty$. Since each point $\psi_i(\widetilde\gamma(t_i))$ lies on a lift of $\widetilde\gamma$ and the union of the lifts of $\gamma$ is a closed set, $\widetilde q_0$ must lie on a lift of $\gamma$. Then, so do all the points $\widetilde q_i$. However, all these points lie on $T$ and our construction of $T$ ensures that the only lift of $\gamma$ that intersects $T$ is $\widetilde\gamma$. Hence each covering transformation $\psi_i$ must translate $\widetilde\gamma$, which is impossible since the geodesic $\gamma$ is not closed (and it is clear that $\psi_i$ is not the identity for large $i$).
\end{proof}
\vspace{.1cm}

We would like to say a little more. It is evident from the above proof that the ray $\gamma|_{[t_0,\infty)}$ must have self intersections for any $t_0$ (and the same is true for any ray of the form $\gamma|_{(-\infty,t_0]}$). Moreover, we have the following corollary about the angles at intersections:
\vspace{.1cm}

\begin{corollary}\label{selfintcor}
Let $\gamma$ be a non-closed geodesic on $M$ along which the curvature is everywhere $0$. Then for any small enough $\alpha>0$ there is a sequence of times $t_i \to \infty$ such that  $\gamma$ crosses itself at time $t_i$ and at an angle greater than $\alpha$.
\end{corollary}

\begin{proof} There is a geodesic $\beta$ such that $\beta'(0) \in \omega(\gamma'(0))$. It must be a non-closed flat geodesic by Lemma \ref{flat closed geodesic}. We can find $\alpha > 0$ and $T,T' > 0$ such that $\beta(T) = \beta(T')$ and the vectors $\beta'(T)$ and $\beta'(T')$ make an angle at least $2\alpha$.  Since $\beta'(0) \in \omega(\gamma'(0))$, there is a sequence $t_i \to \infty$ such that $\gamma'(t_i) \to \beta'(0)$. Then $\gamma'(t_i + T) \to \beta'(T)$ and
$\gamma'(t_i + T') \to \beta'(T')$ as $T \to \infty$. Since $M$ is a surface, $\gamma$ will have an infinite sequence of self-intersections at angle at least $\alpha$ near $\beta(T) = \beta(T')$.
\end{proof}
\vspace{.3cm}

\section{Proof of main theorems}\label{section 5}

\subsection{Proof of Theorem \ref{B}}
Now we assume that $\Lambda \cap (\text{Per\ } (g^t))^c \neq \emptyset$, in other words, there exists an aperiodic orbit $\mathcal{O}(x)$ in $\Lambda$. We will construct the points $y, z\in\Lambda$ as stated in Theorem \ref{B} starting from $\mathcal{O}(x)$ based on the expansivity property of $x$. A first observation is that we can always find two arbitrarily nearby points on the orbit $\mathcal{O}(x)$. We state this result in the following lemma (see Lemma 3.3 in \cite{Wu} for the proof).
\vspace{.1cm}

\begin{lemma}\label{closelemma}
For any $k\in \mathbb{N}$, there exist two sequences $t_k \to +\infty,$ and $t'_k\to +\infty$, such that $t'_k-t_k\to +\infty$ and
$$d(x_k, x'_k)< \frac{1}{k}, \quad \text{\ where\ \ } x_k=g^{t_k}(x), \ x'_k=g^{t'_k}(x).$$
\end{lemma}
\vspace{.1cm}

For each pair $x_k, x'_k$ with large enough $k$, we can check the expansivity in the positive direction of the flow by using the idea in the proof of Proposition 3.4 in \cite{Wu}. In fact, the expansivity in one direction (either positive or negative) of the flow is sufficient for our purpose.
\vspace{.1cm}

\begin{proposition} \label {prop}
Fix an arbitrary small $\epsilon_0 >0$. Then there exists $s_k\to +\infty$ or $s_k\to -\infty$, such that
$$d(g^{s_k}(x_k), g^{s_k}(x'_k))=\epsilon_0,$$
$$\text{\ and\ \ }d(g^s(x_k), g^s(x'_k))< \epsilon_0, \ \forall\ 0 \leq s< s_k \ \text{or\ }\forall\ s_k<s\leq 0\text{\ respectively}.$$
\end{proposition}
\begin{proof}
Assume the contrary. Then $x$ does not have expansivity property. By Lemma \ref{expansivity}, $x$ is tangent to a flat strip. Then by Lemma \ref{flat strip closed}, $x$ must be periodic. This contradicts to the assumption $x\in (\text{Per\ }(g^t))^c$.
\end{proof}
\vspace{.1cm}

Without loss of generality, we suppose that $s_k \to +\infty$ in the remainder of the paper. For the case $s_k \to -\infty$, everything remains true by a slight modification.
\vspace{.1cm}

\begin{proposition}\label{prop2}
For an arbitrary small $\epsilon_0 >0$, there exist $a, b\in \Lambda \cap (\text{Per\ }(g^t))^c$ such that
\begin{equation}\label{1}
 d(a, b)=\epsilon_0,
\end{equation}
\begin{equation}\label{2}
   d(g^t(a), g^t(b))\leq \epsilon_0, \ \forall t<0,
\end{equation}
\begin{equation}\label{3}
a \notin \mathcal{O}(b),
\end{equation}
\begin{equation}\label{4}
a \in W^u(b).
\end{equation}

\end{proposition}

\begin{proof}

We apply Proposition \ref{prop}. Pick a subsequence $k_i\to +\infty$ such that both of the sequences $\{g^{s_{k_i}}(x_{k_i})\}$ and $\{g^{s_{k_i}}(x'_{k_i})\}$ converge. Let
$$a:=\lim_{k_i\to +\infty}g^{s_{k_i}}(x_{k_i}),\ \text{\ and\ }\ b:=\lim_{k_i\to +\infty}g^{s_{k_i}}(x'_{k_i}).$$
Then $d(a, b)=\lim_{k_i\to +\infty}d(g^{s_{k_i}}(x_{k_i}),g^{s_{k_i}}(x'_{k_i}))=\epsilon_0.$ We get \eqref{1}.\\

For any $t<0$, since $0<s_{k_i}+t<s_{k_i}$ for large $k_i$, one has
$$d(g^t(a), g^t(b))=\lim_{k_i\to +\infty}(d(g^{s_{k_i}+t}(x_{k_i}),g^{s_{k_i}+t}(x'_{k_i}))) \leq \epsilon_0.$$
Hence we get \eqref{2}.\\

Next assume that $a$ is periodic. Since
$$\lim_{k_i\to +\infty} g^{t_{k_i}+s_{k_i}}(x)=\lim_{k_i\to +\infty}g^{s_{k_i}}(x_{k_i})=a,$$ then $x$ is periodic by Lemma \ref{flat closed geodesic}. This is a contradiction. So $a\in (\text{Per\ }(g^t))^c$. Similarly $b\in (\text{Per\ }(g^t))^c$. Thus $a, b\in \Lambda \cap (\text{Per\ }(g^t))^c$.\\

Now we prove \eqref{3}, i.e., $a \notin \mathcal{O}(b)$. For a simpler notation, we write
$$\lim_{k\to +\infty}g^{s_{k}}(x_{k})=a, \text{\ and\ }\lim_{k\to +\infty}g^{s_{k}}(x'_{k})=b.$$
The geodesics $\gamma_{x_k}(t), \gamma_{x'_k}(t)$ on $M$ can be lifted to $\widetilde{\gamma}_k, \widetilde{\gamma}'_k$ respectively on $\widetilde{M}$ in the way such that $d(x_k, x'_k)<\frac{1}{k}$,
$d(y_k,y'_k)=\epsilon_0$ where $y_k=g^{s_k}(x_k)$, $y'_k=g^{s_k}(x'_k)$, and moreover $y_k \to a$, $y'_k \to b$. Here we use a same notation for the lift of a point since no confusion is caused. Then $\widetilde{\gamma}_k$ converges to $\widetilde{\gamma}=\widetilde{\gamma}_a$, $\widetilde{\gamma}'_k$ converges to $\widetilde{\gamma}'=\widetilde{\gamma}_b$ and $d(a, b)= \epsilon_0$. See Figure 2 (we use a same notation for a vector and its footpoint).\\

\begin{figure}[h]
\centering
\setlength{\unitlength}{\textwidth}

\begin{tikzpicture}
\draw (0,0) circle (2);

\draw (0,-2) to [out=92,in=275] (-0.5,1.92)node[above] {$\widetilde{\gamma}_k$};
\draw (0.2,-1.98)to [out=91,in=268] (0.5, 1.92)node[above] {$\widetilde{\gamma}'_k$};

\coordinate [label=left:$x_k$] (x_k) at (-0.04,-1.5);
\coordinate [label=right:$x'_k$] (x'_k) at (0.2,-1.5);
\coordinate [label=left:$y_k$] (y_k) at (-0.39,1);
\coordinate [label=right:$y'_k$] (y'_k) at (0.45,1.2);
\coordinate [label=right:$z_k$] (z_k) at (0.45,.9);

\draw (y_k) to [out=25,in=175] node[above ] {$\epsilon_0$}(y'_k);
\draw (x'_k) to [out=100,in=290] (y_k);
\draw (y_k) to [out=-10,in=175] (z_k);
\draw [right angle symbol={y_k}{z_k}{x'_k} size=.01];

\fill (x_k) circle (1.5pt) (x'_k) circle (1.5pt) (y_k)  circle (1.5pt) (y'_k)  circle (1.5pt)
(z_k)  circle (1.5pt);

\end{tikzpicture}
\caption[]{Proof of $\widetilde{\gamma} \neq \widetilde{\gamma}'$}

\end{figure}
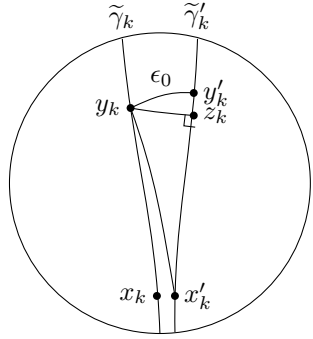

First we show that $d(y_k, \widetilde{\gamma}_k')$ is bounded away from $0$. Write $d_k:= d(y_k, \widetilde{\gamma}'_k)=d(y_k, z_k)$, $l_k:=d(y_k, x'_k)$, $b_k:= d(x'_k, z_k)$, and $b'_k:=d(z_k, y'_k)$. And we already know that $d(x'_k, y'_k)=s_k$. Suppose that $d_k \to 0$ as $k\to +\infty$. By the triangle inequality, $\lim_{k\to +\infty}(l_k-b_k)=0$. Since $\lim_{k\to +\infty}(l_k-s_k) \leq \lim_{k \to+\infty}d(x_k, x'_k)=0$, we have that $\lim_{k\to+\infty}b'_k=\lim_{k\to +\infty}|(l_k-b_k)-(l_k-s_k)|=0$. But the triangle inequality implies $\epsilon_0 \leq d_k+b_k' \to 0 $, which is a contradiction. Now $\widetilde{\gamma} \neq \widetilde{\gamma}'$ follows from $d(a, \widetilde{\gamma}')=\lim_{k\to +\infty}d(y_k, \widetilde{\gamma}'_k) \geq d_0$ for some $d_0 >0$.\\

Next we suppose there exists a deck transformation $\phi$ such that $\phi(\widetilde{\gamma})=\widetilde{\gamma}'$. See Figure 3. Observe that $\widetilde{\gamma}(-\infty)=\widetilde{\gamma}'(-\infty)$ since $d(g^t(a), g^t(b))\leq \epsilon_0$, $\forall t<0$. Let $\widetilde{\gamma}_0$ be the closed geodesic such that $\phi(\widetilde{\gamma}_0)=\widetilde{\gamma}_0$. Then $\widetilde{\gamma}(-\infty)=\widetilde{\gamma}_0(-\infty)$. By Lemma \ref{flat closed geodesic}, $\widetilde{\gamma}$ is a closed geodesic, i.e., $a$ is a periodic point. We arrive at a contradiction. Hence for any deck transformation $\phi$, $\phi(\widetilde{\gamma})\neq\widetilde{\gamma}'$. So $a \notin \mathcal{O}(b)$, and we get \eqref{3}.\\

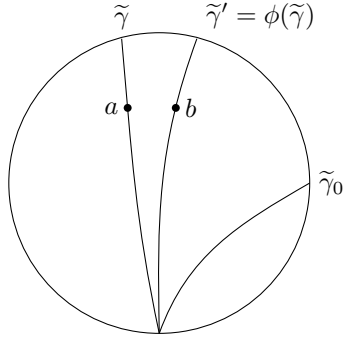
\begin{figure}[h] \label{figure3}
\centering
\setlength{\unitlength}{\textwidth}

\begin{tikzpicture}
\draw (0,0) circle (2);

\draw (0,-2) to [out=102,in=275] (-0.5,1.92)node[above] {$\widetilde{\gamma}$};
\draw (0,-2)to [out=91,in=250] (0.5, 1.92)node[above right] {$\widetilde{\gamma}'=\phi(\widetilde{\gamma})$};
\draw (0,-2) to [out=70,in=210] (2,0)node[right] {$\widetilde{\gamma}_0$};

\coordinate [label=left:$a$] (a) at (-0.42,1);
\coordinate [label=right:$b$] (b) at (0.22,1);
\fill (a) circle (1.5pt) (b) circle (1.5pt);
\end{tikzpicture}
\caption[]{Proof of $\phi(\widetilde{\gamma}) \neq \widetilde{\gamma}'$}
\end{figure}

At last, if $a \notin W^u(b)$, we can replace $a$ by some $a' \in \mathcal{O}(a)$, $b$ by some $b' \in \mathcal{O}(b)$ such that  $a'\in W^u(b')$ and the above three properties still hold for a different $\epsilon_0$. We get \eqref{4}.
\end{proof}
\vspace{.1cm}

\begin{proof}[Proof of Theorem \ref{B}]
We apply Proposition \ref{prop2}. Let $y=-a$, $z=-b$. Then $y, z \in \Lambda \cap (\text{Per\ }(g^t))^c$, $d(g^t(y), g^t(z))\leq \epsilon_0,\ \forall t>0$, $z \notin \mathcal{O}(y)$ and $y \in W^s(z)$.\\

If $\epsilon_0$ is small enough, we can lift geodesics $\gamma_y(t)$ and $\gamma_z(t)$ to $\widetilde{\gamma}_y(t)$ and $\widetilde{\gamma}_z(t)$ respectively on $\widetilde{M}$, such that $d(\widetilde{\gamma}_y(t),\widetilde{\gamma}_z(t)) \leq \epsilon_0$ for any $t >0$ and $y\in \widetilde{W}^s(z)$. Suppose $\lim_{t\to +\infty} d(\widetilde{\gamma}_y(t), \widetilde{\gamma}_z(t))=\delta >0$. Then by Lemma \ref{boundary}, $\widetilde{\gamma}_y(t)$ and $\widetilde{\gamma}_z(t)$ converge to the boundary of a flat strip. Hence $y$ and $z$ are periodic by Lemma \ref{flat closed geodesic}. A contradiction. So we have $\lim_{t\to +\infty} d(\widetilde{\gamma}_y(t), \widetilde{\gamma}_z(t))=0$. Hence $$d(g^t(y), g^t(z))\to 0,  \ \ \text{as\ } t\to +\infty.$$
\end{proof}
\vspace{.1cm}

\subsection{Proof of Theorem \ref{A}}
In the proof of Theorem \ref{A}, an argument similar to the one in Proposition \ref{prop2} is used.
\vspace{.1cm}

\begin{proof}[Proof of Theorem \ref{A}]
Suppose that $\Lambda \subset \text{Per\ }(g^t)$. We will prove that if $x\in \Lambda$, then $x$ is tangent to an isolated closed flat geodesic or to a flat strip.\\

Assume the contrary to Theorem \ref{A}. Then there exists a sequence of different vectors $x'_k \in \Lambda$ such that $\lim_{k\to +\infty}x'_k=x$ for some $x\in \Lambda$. Here different $x'_k$ are tangent to different isolated closed geodesics or to different flat strips, and $x$ is tangent to an isolated closed geodesic or to a flat strip. For large enough $k$, we suppose that $d(x'_k, x) <\frac{1}{k}$. Fix a small number $\epsilon_0 >0$. It is impossible that $d(g^t(x'_k), g^t(x)) \leq \epsilon_0$, $\forall t>0$. For otherwise, $\widetilde{\gamma}_{x'_k}(t)$ and $\widetilde{\gamma}_x(t)$ are positively asymptotic closed geodesics. They must be tangent to a common flat strip by Lemma \ref{boundary} and Lemma \ref{flat closed geodesic0}. This is impossible since different $x'_k$ are tangent to different isolated closed geodesics or to different flat strips. Hence there exists a sequence $s_k\to +\infty$ such that
$$d(g^{s_k}(x'_k), g^{s_k}(x))=\epsilon_0,$$
and
$$d(g^{s}(x'_k), g^{s}(x))\leq \epsilon_0, \ \forall 0\leq s <s_k.$$
\vspace{.1cm}

Let $y_k:=g^{s_k}(x)$ and $y'_k:=g^{s_k}(x'_k)$. Without loss of generality, we suppose that $y_k\to a$ and $y'_k\to b$. A similar proof as in Proposition \ref{prop2} shows that $d(a, b)=\epsilon_0$ and $d(g^t(a), g^t(b)) \leq \epsilon_0$,  $\forall t\leq 0$. Replacing $x, x_k'$ by $-x, -x_k'$ respectively and applying the same argument, we can obtain two points $a^-,b^-$ such that $d(a^-, b^-)=\epsilon_0$ and $d(g^t(a^-), g^t(b^-)) \leq \epsilon_0$, $\forall t\leq 0$. Then we have the following three cases:
 \begin{enumerate}
 \item $\lim_{t\to \infty}d(g^t(-a), g^t(-b))=0$. By Lemma \ref{flat closed geodesic0}, $-a=-b$. This contradicts to $d(a,b)=\epsilon_0$.
 \item $\lim_{t\to \infty}d(g^t(-a^-), g^t(-b^-))=0$. Also by Lemma \ref{flat closed geodesic0}, $-a^-=-b^-$. This contradicts to $d(a^-,b^-)=\epsilon_0$.
 \item $\lim_{t\to \infty}d(g^t(-a), g^t(-b))>0$ and $\lim_{t\to \infty}d(g^t(-a^-), g^t(-b^-))>0$.
\end{enumerate}
\vspace{.2cm}

For case (3), by Lemmas \ref{boundary} and \ref{flat closed geodesic}, $\gamma_{a}$ and $\gamma_{x}$ coincide. And moreover, $\gamma_{x}$ and $\gamma_b$ are boundaries of a flat strip of width $\delta_1>0$. Similarly, $\gamma_{x}$ and $\gamma_{b^-}$ are boundaries of a flat strip of width $\delta_2>0$. We claim that these two flat strips lie on the different sides of $\gamma_x$. Indeed, we choose $\epsilon_0$ small enough and consider the $\epsilon_0$-neighborhood of the closed geodesic $\gamma_x$ which contains two regions lying on the different sides of $\gamma_x$. By the definition of $b$ and $b^-$, they must lie in different regions as above. This implies the claim. \\

In this way we get a flat strip of width $\delta_1+\delta_2$ and $x$ is tangent to the interior of this flat strip. Since $g^{s_k}(x'_k)\to b$, we can repeat all the arguments above to $b, g^{s_k}(x'_k)$ instead of $x, x_k'$. Then either we are arriving at a contradiction as in case (1) or case (2) and we are done, or we get a flat strip of width greater than $\delta_1+\delta_2$ and $b$ is tangent to the interior of the flat strip. But we can not enlarge a flat strip repeatedly in this way on a compact surface $M$. So we are done with the proof.

\end{proof}
\vspace{.3cm}

\subsection{Proof of Theorem \ref{Main}}
We shall prove Theorem \ref{Main} by showing that the second one of the dichotomy cannot happen, if $\{p\in M: K(p)<0\}$ has at most finitely many connected components. The proof is an adaption of the one of Theorem 1.6 in \cite{Wu} to surfaces without focal points. Moreover, we fix a gap in that proof, which is pointed out by Keith Burns, and the argument here is also due to him. Corollary \ref{selfintcor} is used in the argument.
\vspace{.1cm}

\begin{proof}[Proof of Theorem \ref{Main}]
Suppose $\Lambda \cap (\text{Per\ }(g^t))^c \neq \emptyset$. Consider the two points $y$ and $z$ given by Theorem \ref{B}. We lift the geodesics $\gamma_y(t)$ and $\gamma_z(t)$ to the geodesics in the universal cover $\widetilde{M}$, which are denoted by $\widetilde{\gamma}_1$ and $\widetilde{\gamma}_2$ respectively.\\

Consider the connected components of $\{p\in \widetilde{M}: K(p)<0\}$ on $\widetilde{M}$ and we want to see how they distribute inside the ideal triangle bounded by $\widetilde{\gamma}_1$ and $\widetilde{\gamma}_2$. Since $\widetilde{\gamma}_1$ and $\widetilde{\gamma}_2$ are flat geodesics, no connected component intersects $\widetilde{\gamma}_1$ or $\widetilde{\gamma}_2$. We also claim that the radii of inscribed disks inside these connected components are bounded away from $0$. Indeed, if this is not true, then there exists an isometry between the inscribed disk with very small radius inside a connected component, and an inscribed disk inside a connected component of $\{p\in M: K(p)<0\}$ on the base space $M$. This is impossible because the number of the connected components of $\{p\in M: K(p)<0\}$ is finite, and therefore the radii of their inscribed disks are bounded away from $0$. The claim follows.\\

According to Corollary \ref{selfintcor}, there exists a sequence of times $t_i \to \infty$ and isometries $\phi_i$ of $\widetilde{M}$, such that  $\phi_i(\widetilde\gamma_1)$ crosses $\widetilde\gamma_1$ at the point $\widetilde\gamma_1(t_i)$ and at an angle greater than a small enough constant $\alpha>0$. Recall that $d(\widetilde{\gamma}_1(t), \widetilde{\gamma}_2(t))\to 0$ as $t \to +\infty$. Hence $\phi_i(\widetilde\gamma_1)$ also crosses $\widetilde\gamma_2$ near the point $\widetilde\gamma_1(t_i)$ for $i$ large enough. Let $D_i$ be the region in $\widetilde M$ enclosed by $\widetilde\gamma_1$, $\widetilde\gamma_2$, $\phi_i(\widetilde\gamma_1)$ and $\phi_{i+1}(\widetilde\gamma_1)$. Every connected component of $\{p\in \widetilde{M}: K(p)<0\}$ must lie inside at most one single $D_i$. Since $d(\widetilde{\gamma}_1(t), \widetilde{\gamma}_2(t))\to 0$ as $t \to +\infty$, we know that the radius of the inscribed disk inside $D_i$ goes to zero as $i\to \infty$. Combining with the above claim, it is clear that the connected components of $\{p\in \widetilde{M}: K(p)<0\}$ cannot approach $w$ inside the ideal triangle. See Figure 4.

\begin{figure}[h]

\centering
\setlength{\unitlength}{\textwidth}

\begin{tikzpicture}
\draw (0,0) circle (2);

\coordinate [label=left:$y_{t_0}$] (y_{t_0}) at (-.42,0);
\coordinate [label=right:$z_{t_0}$] (z_{t_0}) at (.37,-.3);
\coordinate [label=above right:$w$] (w) at (1,1.73);
\fill (z_{t_0}) circle (1.5pt) (y_{t_0}) circle (1.5pt) (w) circle (1.5pt);

\draw (-1.73,-1) node[below left] {$\widetilde{\gamma}_1$} to [out=35,in=240] (1,1.73);
\draw (0,-2) node[below] {$\widetilde{\gamma}_2$} to [out=80,in=250] (1,1.73);
\draw (y_{t_0}) to [out=340,in=170] (z_{t_0});

\draw (-.7,-1.4) circle (0.3);
\draw (-.4,-.93) circle (.18);
\draw (-.2,-.6) circle (.12);
\draw (-.01,-.35) circle (.08);

\end{tikzpicture}
\caption[]{Proof of Theorem \ref{Main}}

\end{figure}
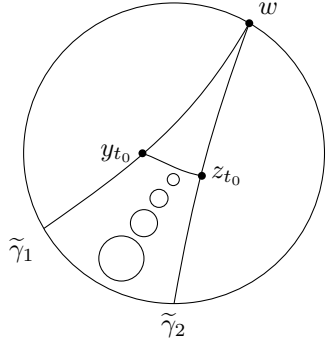

So there exists a $t_0 >0$ with $y_{t_0}:=g^{t_0}(y)$ and $z_{t_0}:=g^{t_0}(z)$,  such that the infinite triangle $z_{t_0}y_{t_0}w$ is a flat region. Then $d(g^t(y), g^t(z))\equiv d(y_{t_0}, z_{t_0})$ for all $t \geq t_0$. Indeed, if we construct a geodesic variation between $\widetilde{\gamma}_1$ and $\widetilde{\gamma}_2$, then the Jacobi fields are constant for $t\geq t_0$ since $K\equiv 0$. Thus $d(\widetilde{\gamma}_1(t), \widetilde{\gamma}_2(t))$ is constant when $t \geq t_0$. We get a contradiction since $d(g^t(y), g^t(z))\to 0$ as $t \to +\infty$ by Theorem \ref{B}.\\

Finally we can conclude that $\Lambda \subset \text{Per\ }(g^t)$. In particular the geodesic flow is ergodic by Theorem \ref{A} and Pesin Theorem (Theorem \ref{pesin}).
\end{proof}
\ \
\\[-2mm]
\textbf{Acknowledgement.}
We would like to thank Keith Burns for his valuable suggestions. Some part of this paper greatly benefits from a fruitful discussion with him.

The first author would also like to thank Federico Rodriguez Hertz for his help and valuable comments. The second author is partially supported by NSFC under Grant No.11301305. The third author is partially supported by  the State Scholarship Fund from China Scholarship Council (CSC).

\end{document}